\documentclass[11pt,twoside,a4paper]{article}

\author{Shai Sarussi}
 \title{Extensions of integral domains and quasi-valuations}
\date{}

\usepackage{amsmath,amsthm,amssymb}
\usepackage{verbatim}

\begin{document}

\newtheorem{thm}{Theorem}[section]
\newtheorem{cor}[thm]{Corollary}
\newtheorem{lem}[thm]{Lemma}
\newtheorem{prop}[thm]{Proposition}
\newtheorem{ax}{Axiom}

\theoremstyle{definition}
\newtheorem{defn}[thm]{Definition}

\theoremstyle{remark}
\newtheorem{rem}[thm]{Remark}
\newtheorem{ex}[thm]{Example}
\newtheorem*{notation}{Notation}

\newcommand{\qv}{{quasi-valuation\ }}


\maketitle

\begin{abstract} {Let $S$ be an integral domain with field of fractions $F$ and let $A$ be an $F$-algebra
having an $S$-stable basis.
We prove the existence of an $S$-subalgebra $R$ of $A$ lying over $S$ whose localization with respect to $S$ is $A$
(we call such $R$ an $S$-nice subalgebra of $A$).
We also show that there is no such minimal $S$-nice subalgebra of $A$.
Given a valuation $v$ on $F$ with a corresponding valuation domain $O_v$, and an $O_v$-stable basis of $A$ over $F$,
we prove the existence
of a quasi-valuation on $A$ extending $v$ on $F$. Moreover, we prove the existence of an infinite decreasing chain of quasi-valuations on $A$,
all of which extend $v$.
Finally, we present applications for the above existence theorems; for example, we show that
if $A$ is commutative and $\mathcal C$ is any chain of prime ideals of $S$, 
then there exists an $S$-nice subalgebra of $A$, having a chain of prime ideals covering $\mathcal C$.

}\end{abstract} 

\section{Introduction}


Valuation theory has long been a key tool in commutative algebra,
with applications in number theory and algebraic geometry.
It has become a useful tool in the study of finite dimensional division algebras, particulary in
the construction of examples, such as Amitsur's
construction of noncrossed products division algebras. See [Wad]
for a comprehensive survey.

Although valuations provide a powerful tool in studying arithmetic
of fields, it has been difficult to use them in noncommutative settings.
For example, division algebras do not have many valuations and rings with zero divisors
do not have valuations at all.
This has motivated researchers to generalize the notion of valuation.
Attempts to generalize the notion of valuation were made
throughout the last few decades. Morandi's value functions (cf.
[Mor]) and Tignol and Wadsworth's gauges [TW] are examples of such
generalizations. Also see [KZ] for Manis valuations and
PM-valuations, and [Co], [Hu], and [MH] for pseudo-valuations.
See [Sa1] for a brief discussion on these other related theories.

Another approach was initiated recently by the author in developing the notion of a quasi-valuation.
A {\it quasi-valuation} on a ring $R$ is a function
$w : R \rightarrow ~M \cup \{  \infty \}$, where $M$ is a totally
ordered abelian monoid, to which we adjoin an element $\infty$
greater than all elements of $M$, and $w$ satisfies the following
properties:

(B1) $w(0) = \infty$;

(B2) $w(xy) \geq w(x) + w(y)$ for all $x,y \in R$;

(B3) $w(x+y) \geq \min \{ w(x), w(y)\}$ for all $x,y \in R$.

In [Sa1] we mainly developed the theory of quasi-valuations on finite dimensional field extensions and we were able to answer questions
regarding the structure of rings using quasi-valuation theory. More precisely, for a given valuation $v$ on a field $F$, a corresponding valuation domain $O_v$, and a finite field extension $E$, we studied quasi-valuations on $E$ extending $v$ on $F$.
We showed that every such quasi-valuation is dominated by some valuation extending v, and presented several generalizations
of results from valuation theory (see [Sa1, sections 6 and 4]).
We also studied the quasi-valuation rings; namely, the set of all elements of $E$ with values greater or equal to zero.
We proved that the prime spectra of $O_v$ and its quasi-valuation ring are intimately connected.
In addition, a one-to-one correspondence was obtained
between exponential quasi-valuations and integrally closed quasi-valuation rings.
Most importantly, we constructed the filter quasi-valuation, for any algebra over a valuation domain,
and showed that if $A$ is an $F$-algebra and $R$ is an $O_v$-subalgebra of $A$ lying over $O_{v}$ then there exists
a quasi-valuation on $R \otimes_{O_{v}} F$ (called the filter quasi-valuation) extending $v$ on $F$ such that the quasi-valuation ring is equal to $R$ (under the identification of $R$ with $R \otimes_{O_{v}} 1$). In particular, if $R$ is an $O_v$-subalgebra of $A$ lying over $O_{v}$ such that $RF=A$ then there exists
a quasi-valuation on $A$ extending $v$ on $F$.
However, the existence of such subalgebras was not clear and was not proven. Existence theorems of such algebras and others of greater generality will be presented in this paper.

In [Sa2] we studied the structures of algebras over valuation domains using quasi-valuation theory.
We generalized some of the results of [Sa1] and presented additional connections between
a valuation domain $O_v$ and an $O_v$-algebra. We related the prime spectrum of a valuation domain to the prime spectrum of a (not necessarily commutative) algebra over it. We studied the classical lifting conditions of ``lying over" (LO), ``incomparability" (INC), ``going down" (GD) and ``going up" (GU) in ring extensions, as well as a subtler condition called ``strong going between" (SGB), and saw how quasi-valuations play a key role in the situation under discussion (see [Sa2, section 1.2] for the definitions of LO, INC, GD, GU and SGB). Specifically, we presented a necessary and sufficient condition for an $O_{v}$-algebra to satisfy LO over $O_{v}$.
We proved that if $R$ is a torsion-free algebra over $O_{v}$ such that $[R \otimes_{O_{v}}F:F] < \infty $,
then $R$ satisfies INC over $O_{v}$; as a result, we obtained an upper and a lower bound on the size
of the prime spectrum of $R$. Then, we showed that if $R$ is a torsion-free algebra over $O_{v}$ then $R$ satisfies GD over $O_{v}$;
and concluded that any algebra over a commutative valuation ring satisfies SGB over it.
We deduced that a torsion-free algebra over $O_{v}$ satisfies GGD (generalized going down) over $O_{v}$.
Moreover, a sufficient condition for a \qv ring to satisfy GU over $O_v$ was given.

So, in [Sa1] and in some parts of [Sa2] we assumed that a \qv extending the valuation $v$ exists (or equivalently by [Sa1, Theorem 9.19],
we assumed the existence of an appropriate $O_v$-algebra).
A natural and central question would be then: when does such a \qv exist?
In this paper we show that quasi-valuations extending a valuation $v$ on $F$ exists on any finite dimensional $F$-algebras,
and even more generally, on any $F$-algebra having an $O_v$-stable basis
(see Definition \ref{stable} for the definition of an $O_v$-stable basis).
In fact, we prove a more general theorem and apply
it to quasi-valuation theory. Let us quote our first existence theorem: let $S$ be an integral domain which is not a field, let
$F$ be its field of fractions, and let $A$ be an $F$-algebra containing an $S$-stable basis;
then there exists an $S$-subalgebra of $A$ which lies over $S$
and whose localization with respect to $S$ is $A$.
This existence theorem is then followed by several other existence theorems;
for example, we prove that any such $S$-subalgebra of $A$ is not
minimal with respect to inclusion, and we also prove the existence of such $S$-subalgebras of $A$
containing a given ideal of $A$.
Returning to the assumption that $v$ is a valuation on $F$, using the construction of the filter \qv (see section 2),
we deduce the existence of an infinite descending chain of quasi-valuations on $A$, all of which
extend $v$ on $F$. Finally, as applications to the above mentioned existence theorems, in Proposition \ref{every chain is covered} and Theorem \ref{every chain is covered commutative case} we use results from valuation theory and results from [Sa2] about quasi-valuations to prove two interesting propositions
regarding the prime spectra of integral domains and certain algebras over them.

In this paper the symbol $\subset$ means proper inclusion and the symbol $\subseteq$ means inclusion
or equality.

\section{Previous results - construction of the filter \qv}

In this section, for the reader's convenience, we recall from [Sa1] the main steps
in constructing the filter quasi-valuation. For further details
and proofs, see [Sa1, section 9].

The first step is to construct a value monoid, constructed from
the value group of the valuation. We call this value monoid the
cut monoid. We start by reviewing some of the basic
notions of cuts of ordered sets. For further information
on cuts see, for example, [FKK] or [Weh].

\begin{defn} Let $T$ be a totally ordered set. A subset $U$ of $T$ is
called initial if for every $\gamma \in U$ and $\alpha \in T$, if
$\alpha \leq \gamma$ then $\alpha \in U$. A cut $\mathcal
A=(\mathcal A^{L}, \mathcal A^{R})$ of $T$ is a partition of $T$
into two subsets $\mathcal A^{L}$ and $\mathcal A^{R}$, such that,
for every $\alpha \in \mathcal A^{L}$ and $\beta \in \mathcal
A^{R}$, $\alpha<\beta$. \end{defn}

 The set of all cuts $\mathcal A=(\mathcal A^{L}, \mathcal A^{R})$
 of the ordered set $T$ contains the two cuts $(\emptyset,T)$ and
 $(T,\emptyset)$; these are commonly denoted by $-\infty$ and
 $\infty$, respectively. However, we do not use the symbols $-\infty$ and
 $\infty$ to denote the above cuts since we
 define a ``different" $\infty$.

Given $\alpha \in T$, we denote
$$(-\infty,\alpha]=\{\gamma \in T \mid \gamma \leq \alpha \}$$ and
$$(\alpha,\infty)=\{\gamma \in T \mid \gamma > \alpha \}.$$
One defines similarly the sets $(-\infty,\alpha)$ and
$[\alpha,\infty)$.

To define a cut we often write $\mathcal A^{L}=U$, meaning the
$\mathcal A$ is defined as $(U, T \setminus U)$ when $U$ is an
initial subset of $T$. The ordering on the set of all cuts of $T$
is defined by $\mathcal A \leq \mathcal B$ iff $\mathcal A^{L}
\subseteq \mathcal B^{L}$ (or equivalently $\mathcal A^{R}
\supseteq \mathcal B^{R}$).

For a group $\Gamma$, subsets $U,U' \subseteq \Gamma$ and $n \in
\Bbb N$, we define

$$U+U'=\{ \alpha+\beta \mid \alpha \in U, \beta \in U' \};$$

$$nU=\{ s_{1}+s_{2}+...+s_{n} \mid s_{1},s_{2},...,s_{n} \in U
\}.$$

Now, for $\Gamma$ a totally ordered abelian group, $\mathcal
M(\Gamma)$ is called the cut monoid of $\Gamma$;
$\mathcal M(\Gamma)$ is a totally ordered abelian monoid.

For $\mathcal A , \mathcal B \in \mathcal M(\Gamma)$, their (left)
sum is the cut defined by $$(\mathcal A + \mathcal B)^{L}=\mathcal
A^{L} + \mathcal B^{L}.$$ The zero in $\mathcal M(\Gamma)$ is the
cut $((-\infty,0],(0,\infty))$.

 For $\mathcal A \in \mathcal M(\Gamma)$ and
$n \in \Bbb N$,  the cut $n\mathcal A$ is defined by $$(n\mathcal
A)^{L}=n\mathcal A^{L}.$$

Note that there is a natural monomorphism of monoids $\varphi :
\Gamma \rightarrow \mathcal M(\Gamma)$ defined in the following
way: for every $\alpha \in \Gamma$, $$\varphi (\alpha) =
((-\infty,\alpha],(\alpha,\infty)) $$

For $\alpha \in \Gamma$ and $\mathcal B \in \mathcal M(\Gamma)$,
we denote $\mathcal B-\alpha$ for the cut $\mathcal B+(-\alpha)$
(viewing $-\alpha$ as an element of $\mathcal M(\Gamma)$).


\begin{defn} Let $v$ be a valuation on a field $F$ with value group
$\Gamma_{v}$. Let $O_{v}$ be the valuation domain of $ v$ and let
$R$ be an algebra over $O_{v}$. For every $x \in R$, the
$O_{v}$-{\it support} of $x$ in $R$ is the set
$$S^{R/O_{v}}_{x}=\{ a \in O_{v} | xR \subseteq aR\}.$$ We suppress
$R/O_{v}$ when it is understood. \end{defn}

For every $A \subseteq O_{v}$ we denote $(v(A))^{\geq 0}=\{ v(a)
\mid a \in A \}$; in particular,

$$(v(S_{x}))^{\geq 0}=\{ v(a) | a \in S_{x} \};$$ the reason for
this notion is the fact that $(v(S_{x}))^{\geq 0}$ is an initial
subset of $\ (\Gamma_{v}) ^{\geq 0}$.

We define $$v(S_{x})=(v(S_{x}))^{\geq 0} \cup (\Gamma_{v})^{<0};$$
and note that $v(S_{x})$ is an initial subset of $\Gamma_{v}$.

Note that if $A$ and $ B$ are subsets of $O_{v}$ such that $A
\subseteq B$ then $v(A) \subseteq v(B)$.

Recall that we do not denote the cut $(\Gamma_{v},\emptyset) \in
\mathcal M(\Gamma_{v}) $ as $\infty$. So, as usual, we adjoin to
$\mathcal M(\Gamma_{v}) $ an element $\infty$ greater than all
elements of $\mathcal M(\Gamma_{v})$; for every $\mathcal A \in
\mathcal M(\Gamma_{v})$ and $\alpha \in \Gamma_{v}$ we define
$\infty+\mathcal A=\mathcal A+\infty=\infty$ and
$\infty-\alpha=\infty$.

 The following theorem holds for arbitrary algebras $R$.

\begin{thm} \label{theorem from Sa on v,F,gammav,Ov} Let $v$ be a valuation on a field $F$ with value group
$\Gamma_{v}$. Let $O_{v}$ be the valuation domain of $\ v$ and let
$R$ be an algebra over $O_{v}$. Let $\mathcal M(\Gamma_{v})$
denote the cut monoid of $\Gamma_{v}$. Then there exists
a quasi-valuation $w:R \rightarrow \mathcal M(\Gamma_{v}) \cup \{
\infty \}$.
\end{thm}

It is shown in the proof of Theorem \ref{theorem from Sa on v,F,gammav,Ov} that
$w$ can be defined by:  $w(0)=\infty$, and
$w(x)=(v(S_{x}),\Gamma_{v} \setminus v(S_{x}))$ for every $0 \neq x \in
 R$; i.e.,  $w(x)^{L}=v(S_{x})$. We call $w$ the
{\it filter quasi-valuation} induced by $(R,v)$.

Let us present several basic properties of the filter quasi-valuation.

\begin{lem} Notation as in Theorem
\ref{theorem from Sa on v,F,gammav,Ov}, assume in addition that
$R$ is torsion free over $O_{v}$; then
$$w(cx)=v(c)+w(x)$$ for every $c \in O_{v}$, $x \in R$.

\end{lem}

We note that even in the case where $R$ is a torsion free algebra
over $O_{v}$, one does not necessarily have $w(c \cdot
1_{R})=v(c)$ for $c \in O_{v}$, despite the fact that
$$w(c \cdot 1_{R}) = v(c)+w( 1_{R})$$ by the previous Lemma. The reason is that $w(
1_{R})$ is not necessarily $0$ (see, for example, [Sa1, Ex. 9.28]).

\begin{rem} Note that if $R$ is a torsion free algebra over
$O_{v}$, then there is an embedding $R \hookrightarrow R
\otimes_{O_{v}}F$; in this case the quasi-valuation on $R
\otimes_{O_{v}}F$ extends the quasi-valuation on $R$. \end{rem}


\begin{lem} Let $v, F, \Gamma_{v}$ and $O_{v}$ be as in Theorem
\ref{theorem from Sa on v,F,gammav,Ov}. Let $R$ be a torsion free
algebra over $O_{v}$, $S$ a multiplicative closed subset of
$O_{v}$, $0 \notin S$, and let $w: R \rightarrow M \cup \{ \infty
\}$ be any quasi-valuation where $M$ is any totally ordered
abelian monoid containing $\Gamma_{v}$ and $w(cx)=v(c)+w(x)$ for
every $c \in O_{v}$, $x \in R$. Then there exists a
quasi-valuation $W$ on $R \otimes_{O_{v}}O_{v}S^{-1}$, extending
$w$ on $R$ (under the identification of $R$ with $R
\otimes_{O_{v}} 1$), with value monoid $ M \cup \{ \infty
\}$.\end{lem}

\begin{thm} Let $v, F, \Gamma_{v}, O_{v}$ and $\mathcal M
(\Gamma_{v})$ be as in Theorem \ref{theorem from Sa on
v,F,gammav,Ov}. Let $R$ be a torsion free algebra over $O_{v}$ and
let $w$ denote the filter quasi-valuation induced by $(R,v)$; then
there exists a quasi-valuation $W$ on $R \otimes_{O_{v}}F$,
extending $w$ on $R$, with value monoid $\mathcal M(\Gamma_{v})
\cup \{ \infty \}$ and $O_{W}=R \otimes_{O_{v}} 1$. \end{thm}

It is not difficult to see that for $W$ as in the previous
Theorem, $(\emptyset, \Gamma_{v}) \notin im(W).$ This $W$ is also
called the filter \qv induced by $(R,v)$.

We conclude the following important theorem (see [Sa1, Theorem 9.34]),

\begin{thm} \label{existence of qv from Sa} Let $v, F, \Gamma_{v}, O_{v}$ and $\mathcal
M(\Gamma_{v})$ be as in Theorem \ref{theorem from Sa on
v,F,gammav,Ov} and let $A$ be an $F$-algebra. Let $R$ be a subring
of $A$ such that $R \cap F =O_{v}$. Then there exists a
quasi-valuation $W$ on $RF$ with value monoid $\mathcal M
(\Gamma_{v})\cup \{ \infty \}$ such that $R=O_{W}$ and {\bf $ W $
extends $v$ (on $F$)}.\end{thm}

\section{Existence Theorems}



In this section $S$ denotes an integral domain which is not a field,
$F$ denotes its field of fractions, and $A \neq F$ is an $F$-algebra, usually taken to contain an $S$-stable basis.

Assuming $A$ contains an $S$-stable basis, we prove the existence of $S$-subalgebras of $A$ which lie over $S$
and whose localizations with respect to $S$ is $A$. we also show that any such $S$-subalgebra of $A$ is not
minimal with respect to inclusion.
Moreover, we prove the existence of such $S$-subalgebras of $A$
containing a given ideal of $A$.

We conclude that for any such
$F$-algebra there exists a \qv on $A$ extending a given valuation $v$
on $F$. In fact, there exists an infinite descending chain of quasi-valuations on $A$, all of which
extend $v$ on $F$.

At the end, we use these existence theorems to show that any chain of prime ideals of $S$
can be covered by a chain of prime ideals of some $S$-subalgebra of $A$ with the properties mentioned above.


Let $v$ be a valuation on $F$ with corresponding valuation domain $O_v$.
By Theorem \ref{existence of qv from Sa}, one can construct a \qv on $A$ when one is given an $O_v$-subalgebra of $A$ lying over $O_v$ and whose localization with respect to
$O_v$ is equal to $A$. Our preliminary goal is to prove the existence of such subalgebras of $A$; but in fact, we prove a more general theorem,
replacing the valuation domain with an integral domain.
Thus, we focus on the existence of such $S$-subalgebras of $A$.
We begin, hence, with an obvious definition.

\begin{defn} 
Let $R$ be an $S$-subalgebra of $A$. We say
that $R$ is an $S$-nice subalgebra of $A$ if
$R \cap F=S$ and $R F=A$. In other words, $R$ is an $S$-nice subalgebra of $A$ if $R$ is lying over $S$ and
the localization of $R$ with respect to $S$ is equal to $A$.
\end{defn}

The following lemma is easy to prove and we shall not prove it here.

\begin{lem} \label{RF=A iff R contains a basis}
Let $R$ be an $S$-subalgebra of $A$. Then, $R  F=A$ iff $R$ contains a basis of $A$ over $F$.
\end{lem}

To be able to study more efficiently infinite dimensional algebras over $F$, we introduce the following definition.

\begin{defn} \label{stable} Let $B$ be a basis of $A$ over $F$. We say that $B$ is $S$-stable
if there exists a basis $C$ of $A$ over $F$ such that for all $c \in C$ and $b \in B$,
one has $cb \in \sum_{y \in B}  Sy$. We also say that $C$ stabilizes $B$ (whenever $S$ is understood).


\end{defn}

\begin{rem} Note that if $B$ is closed under multiplication then $B$ is $S$-stable. Thus, for example,
every free (noncommutative) $F$-algebra with an arbitrary set of generators has an $S$-stable basis;
in particular, every polynomial algebra with an arbitrary set of indeterminates over $F$ has an $S$-stable basis.\end{rem}

The following two remarks are easy to prove. 

\begin{rem} \label{S_i contained in S_2} Let $S_1 \subseteq S_2$ be integral domains with field of fractions $F$. If $B$
is an $S_1$-stable basis of $A$ over $F$, then $B$ is
an $S_2$-stable basis of $A$ over $F$.

\end{rem}

\begin{rem} Let $\{ R_i \}_{i=1}^{n}$ be a finite set of $S$-subalgebras of $A$ such that $R_iF=A$ for all
$1 \leq i \leq n$. Then $R=\bigcap_{i=1}^{n} R_i$ is an $S$-subalgebra of $A$ satisfying $R  F=A$.
In particular, an intersection of a finitely many $S$-nice subalgebras of $A$ is an $S$-nice subalgebra of $A$.

\end{rem}



Although we do not know whether an $S$-stable basis always exists, one can modify a given $S$-stable basis,
as shown in the following lemma.

\begin{lem} \label{if exists basis then exists with 1}

Let $B$ be an $S$-stable basis of
$A$ over $F$ and let $0 \neq x_0  \in A \setminus B$. Then there exists an $S$-stable basis
containing $x_0$. More precisely, there exists an $S$-stable basis of the form $\{ x_0 \} \cup B \setminus \{ b_0 \}$
for some $b_0 \in B$.

\end{lem}

\begin{proof} Let $C$ be a basis that stabilizes $B$.
Let $B'$ be a basis such that $x_0 \in B'$ and $B' \setminus \{ x_0 \}=B \setminus \{ b_0 \}$ for an appropriate $b_0 \in B$ (of course,
one can write $x_0$ as a linear combination of elements of $B$ and take any
element of $B$ appearing in this linear combination).
Take $s_0 \in S$ such that $s_0 b_0 \in \sum_{y \in B'} Sy$.
It is now easy to check that for all $c \in C$ and $b \in B' \setminus \{ x_0 \}$, we have $s_0 c \cdot b \in \sum_{y \in B'} Sy$;
we note by passing that the basis $B''=\{ s_0 b_0 \} \cup B \setminus \{ b_0 \}$ is $S$-stable, since
the basis $C'=\{ s_0 \cdot c \}_{c \in C}$ stabilizes it.
As for $x_0$, for every $s_0 \cdot c \in C'$ there exists $s_c \in S$ such that $s_c s_0 c \cdot x_0 \in \sum_{y \in B'} Sy$.
Thus, the basis $\{ s_c s_0 \cdot c \}_{c \in C}$ stabilizes $B'$.

\end{proof}

By Lemma \ref{if exists basis then exists with 1} and induction we conclude:

\begin{lem} \label{if exists basis then exists with linearly independent set}

Let $B$ be an $S$-stable basis of
$A$ over $F$ and let $C \subseteq A$ be a finite linearly independent set. Then there exists an $S$-stable basis
containing $C$. More precisely, there exists an $S$-stable basis of the form $C \cup B \setminus C_1$
for some finite set $C_1 \subseteq B$.

\end{lem}

In the following proposition we prove the existence of an $S$-subalgebra of $A$ lying over $S$.
\begin{prop} \label{lying over $S$}

Let $B$ be a basis of
$A$ over $F$. 
Let $M=\sum_{b \in B}  Sb$ and $R=\{ x \in
A \mid xM \subseteq M\}$. Then $R$ is an $S$-subalgebra of $A$ satisfying $R \cap F=S$.

\end{prop}

\begin{proof} First note that $M$ is an $S$-submodule
of $A$ and thus $S \subseteq R$ and $R$ is an $S$-subalgebra of $A$.
Now, let $\alpha \in R \cap F$ and take $b \in B $. Then, $\alpha b=\sum_{i=1}^{k}s_i b_i$ for some $s_i \in S$ and $b_i \in B$.
Since the set $\{ b \} \cup \{ b_i\}_{i=1}^{k} \subseteq B$ is linearly independent, $\alpha \in S$.



\end{proof}

Note that in Proposition \ref{lying over $S$}, $B$ is merely a basis of $A$ over $F$
and we do not assume that $B$ is $S$-stable.

\begin{rem} If one takes in Proposition \ref{lying over $S$} a basis $B$ containing $1$ (or any invertible element $u \in S$), then $R \subseteq M$ and $M \cap F = S$.

\end{rem}

\begin{proof} We prove the remark for the case in which $1 \in B$. For the case in which $B$ contains an invertible element $u \in S$ the proof is similar. Now, $R \subseteq  M$ and $S \subseteq  M$ because $1 \in M$. Thus, it is enough to show that $M \cap F \subseteq S$.
So, let $\alpha \in M \cap F$. Since $F$ is the quotient field of $S$, there exists $t \in S $ such that $t \alpha \in S$; clearly, $t \alpha \in M$.
Also, since $\alpha \in M$, one can write
$\alpha=\sum_{i=1}^{k}s_i x_i$ with $s_i \in S$ and $x_i \in B$. Hence, $t \alpha=\sum_{i=1}^{k}ts_i x_i$; but $t \alpha \in S \cap M$ and thus
it can be uniquely written as a linear combination of elements of $B$ in the form $t \alpha=t \alpha \cdot 1$. Therefore, the presentation of $t \alpha$ as $\sum_{i=1}^{k}ts_i x_i$ must be equal to
$t \alpha \cdot 1$; i.e., $k=1$ and $x_1=1$. Consequently, $ \alpha \in S$.

\end{proof}


\begin{rem} In view of the previous remark, $M$ does not necessarily lie over $S$. Indeed, take any basis
$B$ containing an element $\alpha \in F \setminus S$.
\end{rem}

As mentioned above, the existence of an $S$-stable basis is not known in the general case, for any $F$-algebra.
However, when $A$ is finite dimensional over $F$,
the following proposition shows not only the existence of such a basis, but even more so that every basis is $S$-stable.

\begin{prop} \label{RF=A FINITE}
If $A$ is finite dimensional over $F$, then every basis of $A$ over $F$ is $S$-stable.
\end{prop}

\begin{proof} Let $B=\{ x_{1},x_{2},...,x_{n}\}$ be a basis of $A$ over $F$ and let $M=\sum_{i=1} ^{n}
Sx_{i}$. For all $1 \leq i,j \leq n$ one can write $x_i x_j = \sum_{k=1} ^{n} \frac{\alpha_{ijk}}{\beta_{ijk}} x_k$
where $\alpha_{ijk},\beta_{ijk} \in S$. Let $\gamma_{ij}=\prod_{k=1}^{n} \beta_{ijk}$; then $\gamma_{ij}x_i x_j \in M$.
Let $\delta_i= \prod_{j=1}^{n} \gamma_{ij} $; then $C=\{ \delta_1 x_{1},\delta_2 x_{2},...,\delta_n x_{n}\}$ stabilizes $B$.

\end{proof}

We note that the most restrictive assumption that we make in this paper is that $A$ contains an $S$-stable basis.
Therefore, in light of Proposition \ref{RF=A FINITE}, all of the results presented in this paper apply to any
finite dimensional $F$-algebra.

In the following theorem we prove the existence of an $S$-nice subalgebra of $A$.

\begin{thm} \label{Prop existence of S-nice} Let $B$ be an $S$-stable basis of $A$ over $F$. Let
$M=\sum_{b \in B} Sb$ and $R=\{ x \in A \mid xM \subseteq M\}$.
Then $R$ is an $S$-nice subalgebra of $A$.

\end{thm}

\begin{proof} 
Let $C$ be a basis that stabilizes $B$. Then $R$ contains $C$ and thus by Lemma
\ref{RF=A iff R contains a basis}, $R  F=A$.
By Proposition \ref{lying over $S$}, $R$ is lying over $S$. So, $R$ is an $S$-nice
subalgebra of $A$.

\end{proof}

In a less detailed form, Theorem \ref{Prop existence of S-nice} can be restated as follows:

\begin{thm} \label{existence of S-nice} If there exists an $S$-stable basis of $A$ over $F$, then there exists an $S$-nice
subalgebra of $A$.

\end{thm}

Note that in view of Theorem \ref{existence of S-nice} and Proposition \ref{RF=A FINITE}, if $A$ is finite dimensional over $F$
then an $S$-nice subalgebra of $A$ always exists.

The following is a generalization of Theorem \ref{existence of S-nice}.

\begin{thm} \label{I ideal implies exist R nice containing I} Let $I$ be a proper ideal
of $A$. Assume that there exists a basis of $I$ over $F$ that is contained in some $S$-stable basis of $A$ over $F$.
Then there exists an $S$-nice subalgebra $R$ of $A$ such that
$I \triangleleft R$.
\end{thm}

\begin{proof} Let $B_1$ be a basis of $I$ over $F$ and let $B_1 \subset B$ be an $S$-stable basis of $A$ over $F$.
Let $N=I+\sum_{b \in B \setminus B_1} Sb$; $N$ is clearly an $S$-submodule of $A$ and thus $R=\{ x \in A \mid
xN \subseteq N\}$ is an $S$-subalgebra of $A$. We prove that $R \cap F=S$. It is clear that $S
\subseteq R$. On the other hand, let  $\alpha \in R \cap F$ and take $b \in B \setminus B_1$. Then,
$\alpha b \in N$; since $b \notin I$ and $B$ is linearly independent, $\alpha \in S$.
We prove now $R  F=A$. Let $ M=\sum_{b \in B} Sb $ (note
that $M \subseteq N$) and let $R'=\{ x \in A \mid xM \subseteq
M\}$. We prove $R' \subseteq R$. Indeed, Let $x \in R'$, then $$xN = x(I+\sum_{b \in B \setminus B_1}Sb)
= xI + x \sum_{b \in B \setminus B_1} Sb \subseteq I+M =N.$$
Now, by Theorem \ref{Prop existence of S-nice} (or Lemma \ref{RF=A iff R contains a basis}),
$R'F=A$. Therefore, $R  F=A$.
Finally, let $x \in I$; then $xN \subseteq xA \subseteq I \subseteq N$.
Therefore $x \in R$ and the theorem is proved.

\end{proof}

In view of Theorem \ref{I ideal implies exist R nice containing I}, we have the following two remarks.

\begin{rem} Let $I$ be a proper ideal
of $A$. Assume that $A$ contains an $S$-stable basis.
If $I$ has a finite basis over $F$ then, by Proposition \ref{if exists basis then exists with linearly independent set} and Theorem
\ref{I ideal implies exist R nice containing I}, there exists an $S$-nice subalgebra $R$ of $A$ such that
$I \triangleleft R$.

\end{rem}

\begin{rem}

Let $I$ be a proper ideal
of $A$. If $[A:F]< \infty$ then there exists an $S$-nice subalgebra $R$ of $A$ such that
$I \triangleleft R$, by Proposition \ref{RF=A FINITE} and Theorem \ref{I ideal implies exist R nice containing I}.

\end{rem}

Let $E$ be a field with valuation $v$ and corresponding valuation domain $O_{v}$.
It is well known (see, for example, [En, p. 62, Corollary 9.7]) 
that for any field $K \supseteq F$, there exists at least one and often many different valuation domains of $K$ lying over $O_v$.
Let us mention now some other well-known existence theorems regarding extensions of valuation domains.
Three main classes of rings were suggested throughout the years as
the noncommutative version of a valuation ring. These three types
are invariant valuation rings, total valuation rings, and Dubrovin
valuation rings. They are interconnected by the following diagram:
$$\{\text{invariant valuation rings}\}\!\subset\!\{\text {total valuation rings} \}
\!\subset\! \{\text {Dubrovin valuation rings}\}.$$
We shall now define these classes of rings and present their existence theorems.

An invariant valuation ring $V$ of a division ring $D$ is a subring $V$ of $D$ such that for every $a \in D^*$
we have $a \in V$ or $a^{-1} \in V$, and also $aVa^{-1} \subseteq V$.
An invariant valuation ring corresponds to a valuation on $D$, in the usual sense.

Let $D$ be a division ring finite dimensional over its center, $E$. An invariant valuation ring
lying over $O_v$ exists if and only if $v$ extends uniquely to each field $L$ with $E \subseteq L \subseteq D$; in particular, there exists at most
one (but perhaps none) invariant valuation ring lying over $O_v$. See [Wa, Theorem 2.1].

A subring $V$ of a division ring $D$ is called a total valuation ring of $D$
if for every $d \in D^*$, we have $d \in V$ or $d^{-1} \in V$.

Then, there exists a total valuation ring lying over $O_v$ iff the
set $T = \{ d \in D | \ d$ is integral over $O_v \}$ is a ring. When this occurs, there are only
finitely many different total valuation rings lying over $O_v$. See [Wa, Theorem 9.2].
Mathiak has shown (see [Mat, p. 5]) that
they have a valuation-like function whose image is a totally ordered set which is
not a group.

Let $C$ be a simple Artinian ring, a subring $B$ of $C$ is called a Dubrovin
valuation ring of $C$ if $B$ has an ideal $J$ such that $B/J$ is a simple Artinian ring,
and for each $c \in C \setminus B$ there are $b, b' \in B$, such that $cb \in B \setminus J$ and
$b'c \in B \setminus J$.

Let $C$ be a central simple algebra over $E$; i.e., $C$ is a simple $E$-algebra finite dimensional over its center $E$.
Then there exists a Dubrovin valuation ring $B$ of $C$ lying over $O_{v}$. Furthermore,
if $B'$ is another Dubrovin valuation ring of $C$ lying over $O_{v}$,
then there is $c \in C^*$ with $B = cBc^{-1}$. See [Wa, Theorem 10.3].
Morandi (cf. [Mor]) defines a value function which is a quasi-valuation satisfying a few more conditions. Given an integral
Dubrovin valuation ring $B$ of a central simple algebra $C$, Morandi shows that there is a value function
$w$ on $C$ with $B$ as its value ring (the value ring of $w$ is defined as the set of all $x \inC$ such that
$w(x) \geq 0$). Morandi also proves the converse, that if $w$ is a value function on $C$, then the value ring
is an integral Dubrovin valuation ring.

Now, we apply the above results to prove the existence of quasi-valuations on $A$ extending $v$ on $F$.

\begin{thm} \label{there exists qv} Let $v$ be a valuation on $F$ and let $O_v$ be the valuation domain corresponding to $v$.
Let $A$ be an $F$-algebra. If there exists an $O_{v}$-stable basis of $A$ over $F$, then there exists a \qv $w$ on $A$
extending $v$ on $F$.
\end{thm}

\begin{proof} By Theorem \ref{existence of S-nice} there exists an $O_{v}$-nice subalgebra $R$ of $A$.
By Theorem \ref{existence of qv from Sa}, there exists a \qv $w$ on $A$ extending $v$ on $F$ whose corresponding quasi-valuation ring is
$R$.

\end{proof}

\begin{cor} Let $v$ be a valuation on $F$ and let $O_v$ be the valuation domain corresponding to $v$.
Let $A$ be an $F$-algebra and let $I$ be proper ideal of $A$. If there exists a basis of $I$ over $F$ that is contained in some $O_{v}$-stable basis of $A$ over $F$, then there exists a \qv $w$ on $A$
extending $v$ on $F$ such that $w(I) \geq 0$.
\end{cor}

\begin{proof} By Theorem \ref{I ideal implies exist R nice containing I} there exists an $O_v$-nice subalgebra $R$ of $A$
such that $I \lhd R$. The result now follows from Theorem \ref{existence of qv from Sa}.



\end{proof}

Let $I$ be a proper ideal of $A$.
In view of the remarks after Theorem \ref{I ideal implies exist R nice containing I}, if either $A$ is finite dimensional over $F$
or $A$ contains an $O_{v}$-stable basis and $I$ is finite dimensional over $F$, then the conclusion of the previous corollary is valid.

We continue our study of $S$-nice subalgebras of $A$.
The following lemma may be thought of as a going-down lemma for $S$-nice subalgebras.

\begin{lem} \label{going down} Let $S_1 \subseteq S_2$ be integral domains with field of fractions $F$.
Assume that there exists an $S_1$-stable basis of $A$ over $F$.
Let $R$ be an $S_2$-nice subalgebra of $A$. Then there exists an $S_1$-nice subalgebra of $A$, which is contained in $R$. 

\end{lem}
\begin{proof} Let $B$ be an $S_1$-stable basis of $A$ over $F$.
Let $N=\sum_{b \in B} S_1 b$ and let $R_1=\{ x \in A
\mid xN \subseteq N\}$. Then, by Theorem \ref{Prop existence of S-nice}, $R_1$ is an $S_1$-nice subalgebra of $A$.
So, $R_1 \cap R$ is an $S_1$-nice subalgebra of $A$ which is contained in $R$.


\end{proof}

Assuming there exists an $S$-stable basis of $A$ over $F$,
by Theorem \ref{existence of S-nice} and Zorn's lemma it is clear that there exists a maximal (with respect to containment)
$S$-nice subalgebra of $A$. In the following proposition we prove that a minimal $S$-nice subalgebra of $A$ does not exist.

\begin{prop} \label{not minimal} Assume that there exists an $S$-stable basis of $A$ over $F$.
Let $R$ be an $S$-nice subalgebra of $A$. Then there exists an infinite decreasing chain
of $S$-nice subalgebras of $A $ starting from $R$. In particular, a minimal $S$-nice subalgebra of $A $ does not exist.

\end{prop}

\begin{proof} Let $B$ be an $S$-stable basis of $A$ over $F$. Let $y \in R \setminus S$ and let $ s_0 \neq 0$ be a non-invertible element of $S$.
By Lemma \ref{if exists basis then exists with linearly independent set} there exists an $S$-stable basis, say $B_1$, containing
(the linearly independent set) $\{ 1, s_0 y\}$.
Let $M=\sum_{b \in B_1} S b$ and let $R_1=\{ x \in A
\mid xM \subseteq M\}$. By Theorem \ref{Prop existence of S-nice}, $R_1$ is an $S$-nice subalgebra of $A$.
Now, $R_1 \subseteq M$  since $1 \in M$, and $y \notin R_1$ since $y \notin M$.
Thus, $ R \nsubseteq R_1$ and $R_1 \cap R$ is an $S$-nice subalgebra of $A$ which is strictly contained in $R$.

\end{proof}

Assume that there exists an $S$-stable basis of $A$ over $F$.
Let $\mathcal U$ be a maximal chain of $S$-nice subalgebras of $A$; i.e., there is no chain of $S$-nice subalgebras of $A$
that strictly contains $\mathcal U$. Note that by Zorn's lemma there exists such a chain; in fact, by Theorem \ref{existence of S-nice}
every such chain is nonempty and by Proposition \ref{not minimal} every such chain is infinite.
Consider the ring $T=\bigcap_{R \in \mathcal U} R$, which is clearly an $S$-subalgebra of $A$ lying over $S$.
Then, $T$ is not an $S$-nice subalgebra
of $A$. Indeed, otherwise it would be a minimal $S$-nice
subalgebra of $A$, in contradiction to Proposition \ref{not minimal}.
Thus, the localization of $T$ with respect to $S$ is an $S$-subalgebra strictly contained in
$ A$.
It is easy to see that in the following two special cases $T$ is equal to $S$:
(a) whenever $[A:F]=2$; (b) when $A$ is a field and $[A:F]$ is prime. 

Note that in particular, denoting $U'=\{ R \mid R \text{ is an }
S \text{-nice subalgebra of } A \}$, the ring $T'=\cap_{R \in U'} R$ is not an $S$-nice subalgebra
of $A$.

We shall now discuss prime ideals of $S$; in particular, arbitrary chains of prime ideals of $S$.
We will show the existence of $S$-subalgebras of $A$ having a ``sufficiently rich" prime spectrum with respect to $S$.

It is well known that prime ideals can be viewed in terms of algebraic geometry,
where the prime spectrum of a ring is endowed with the usual
Zarisky topology, even when the ring is not commutative; see, for example, [OV, p. 36].
For a brief discussion on examples of rings having infinite chains of prime ideals we refer the reader to [Sa3].

\begin{defn} Let $\mathcal C$ be a chain of prime ideals of $S$ and let $R$ be a faithful $S$-algebra.
Let $\mathcal D$ be a chain of prime ideals of $R$. We say that $\mathcal D$ covers $\mathcal C$ if for every $P \in \mathcal C$
there exists $Q \in \mathcal D$ lying over $P$; namely, $Q \cap S=P$.
\end{defn}


In Proposition \ref {every chain is covered} and Theorem \ref{every chain is covered commutative case} we apply results from valuation theory and from \qv theory. Given any chain of prime ideals $\mathcal C$ of $S$ we prove the existence of $S$-subalgebras of $A$
having a chain of prime ideals covering $\mathcal C$.

\begin{prop} \label{every chain is covered} Assume that there exists an $S$-stable basis of $A$ over $F$.
Let $\mathcal C$ be a chain of prime ideals of $S$.
Then there exists an $S$-subalgebra $R$ of $A$ whose localization with respect to $S$ is $A$ and such that there exists a chain of prime ideals
$\mathcal E$ of $R$ covering $\mathcal C$. In fact, there exists an infinite descending chain of such
$S$-subalgebras of $A$.
\end{prop}

\begin{proof} By [KO, main theorem] there exists a valuation domain $S \subseteq O_v \subset F$ and
a chain of prime ideals
$\mathcal D$ of $O_{v}$ covering $\mathcal C$. By assumption and Remark \ref{S_i contained in S_2}, there exists an $O_v$-stable basis of $A$ over $F$. Thus, by Theorem \ref{existence of S-nice}, there exists
an $O_v$-nice subalgebra $R$ of $A$.
By [Sa2, Corollary 3.9], there exists a chain of prime ideals
$\mathcal E$ of $R$ covering $\mathcal D$; it is clear that $R$ is an $S$-subalgebra of $A$
satisfying $R  F=A$ and $\mathcal E$ lies over $\mathcal C$. The final assertion follows
from Proposition \ref{not minimal} and [Sa2, Corollary 3.9].

\end{proof}

In Proposition \ref{not minimal} we proved that the set of $S$-nice subalgebras of $A$ is quite rich; we shall now prove that
the prime spectra of certain $S$-nice subalgebras of $A$ are rich enough with respect to the prime spectrum of $S$.
This will be done in the case where $A$ is commutative; in this case we can sharpen Proposition \ref{every chain is covered}.

\begin{thm} \label{every chain is covered commutative case} Assume that there exists an $S$-stable basis of $A$ over $F$.
Let $\mathcal C$ be a chain of prime ideals of $S$.
If $A$ is commutative then there exists an $S$-nice subalgebra $R$ of $A$ such that there exists a chain of prime ideals
$\mathcal F$ of $R$ covering $\mathcal C$. In fact, there exists an infinite descending chain of such
$S$-nice subalgebras of $A$.
\end{thm}

\begin{proof} By Proposition \ref{every chain is covered} there exists an $S$-subalgebra $R'$ of $A$ such that
$R'  F=A$ and a chain of prime ideals $\mathcal E$ of $R'$ covering $\mathcal C$.
Clearly $  S \subseteq R' \cap F$. Thus, by Lemma \ref{going down}, there exists an $S$-nice subalgebra $R$ of $A$ which is contained in $R'$.
Since $A$ is commutative, every prime ideal of $R'$ intersects to a prime ideal
of $R$. Thus, $\{Q \ \cap R\}_{Q \in \mathcal E}$ is a chain of prime ideals of $R$ covering $\mathcal C$.
The final assertion follows from Proposition \ref{not minimal}.

\end{proof}

Of course, not every $S$-nice subalgebra of $A$ satisfies the property of the previous theorem. Indeed,
consider any $S$-nice subalgebra of $A$ that does not satisfy LO over $S$.


Let $v$ be a valuation on $F$ with corresponding valuation domain $O_v$, and let $w_1$ and $w_2$ be two quasi-valuations on $A$ extending $v$ on $F$.
We assume that $w_1$ and $w_2$ are comparable; namely, there exists a totally ordered set $M$
such that $w_1(A) \cup w_2(A) \subseteq M \cup \{ \infty\}$.
We write $w_1 \leq' w_2$ if for every $x \in A$, $w_1(x) \leq w_2 (x)$ in $M \cup \{ \infty\}$.

It is clear that if $w_1$ and $w_2$
are both filter quasi-valuations then they are comparable. Thus, the set of all filter
quasi-valuations on $A$ extending $v$ on $F$ is partially ordered by $\leq'$.

Let $R_1 \subseteq R_2$ be $O_v$-nice subalgebras of $A$. Let $w_i$ ($i=1,2)$ be the filter \qv on $A$ induced by $(R_i,v)$.
From the construction of the filter \qv one
can check that $w_1 \leq' w_2$. Now, Let $\mathcal U = \{ R_i \}_{i \in I}$ be a nonempty chain of $S$-nice subalgebras of $A$ and let
$T=\bigcap_{i \in I} R_i$. Let $w_i$ ($i \in I)$ be the filter \qv on $A$ induced by $(R_i,v)$.
Let $w$ denote the filter \qv induced by $(T,v)$.
It is not difficult to see that for every $x \in T  F $ we have $w(x) = \bigcap_{i \in I} w_i(x)  $.
Moreover, if $\mathcal U$ is a maximal chain of $S$-nice subalgebras of $A$ then, as shown above, $T  F $ is a proper
$T$-subalgebra of $A$.
In this case, for all $x \in A$, we have $\bigcap_{i \in I} w_i(x) = \emptyset$ iff $x \in A \setminus T  F$.

By Proposition \ref{not minimal}, Theorem \ref{there exists qv}, and the discussion above we have,

\begin{cor} Let $F$ be a field with valuation $v$ and let $O_v$ be the corresponding valuation domain.
Let $A \neq F$ be an $F$-algebra having an $O_v$-stable basis over $F$.
Then there exists an infinite decreasing chain of quasi-valuations
on $A$ extending $v$ on $F$. Moreover, for any $O_{v} \neq R
\subseteq A$, an $O_v$-subalgebra of $A$ lying over $O_{v}$ there exists an
infinite decreasing chain of quasi-valuations on $R  F$ extending
$v$ on $F$ starting from $w_{R}$; where $w_{R}$ denotes the filter
\qv induced by $(R,v)$.
\end{cor}


In Theorem \ref{not minimal} we
showed that there is no minimal (with respect to inclusion)
$S$-nice subalgebra of $A$; in particular there exists an
infinite descending chain of $S$-nice subalgebras of $A$.
As noted above, by Zorn's Lemma there exists a maximal $S$-nice
subalgebra of $A$. In case $S$ is a valuation domain of $F$ and $A$ is a field then the maximal $S$-nice
subalgebras of $A$ are precisely the valuation domains (whose valuations
extend $v$) of $A$. We shall now show that even in the case of a central
simple $F$-algebra, one can have an infinite ascending chain of
$S$-nice subalgebras of $A$ (even when $S$ is a valuation domain).

\begin{ex} Let $C$ be a non-Noetherian integral domain with field of fractions
$F$. Let $\{ 0 \} \neq I_{1} \subset I_{2} \subset I_{3} \subset
...$ be an infinite ascending chain of ideals of $C$ and
let $A=M_{n}(F)$. Then

\begin{equation*}
 \left(
\begin{array}{cccccccc}
C& C & ... & C & I_{1}  \\
C & C & ... & C & I_{1}  \\
. & . & ... & . & .  \\
. & . & ... & . & .  \\
. & . & ... & . & .  \\
C & C & ... & C & I_{1}  \\
C & C & ... & C & C  \\
\end{array} \right) \subset  \left(
\begin{array}{cccccccc}
C& C & ... & C & I_{2}  \\
C & C & ... & C & I_{2}  \\
. & . & ... & . & .  \\
. & . & ... & . & .  \\
. & . & ... & . & .  \\
C & C & ... & C & I_{2}  \\
C & C & ... & C & C  \\
\end{array} \right)
...
\end{equation*}

 is an infinite accending chain of $C$-nice subalgebras
of $A$.

\end{ex}


\begin{thebibliography}{99}









\bibitem [Bo], N. Bourbaki, \emph {Commutative Algebra}, Chapter 6,
Valuations, Hermann, Paris, 1961.

\bibitem [Co], P. M. Cohn, \emph {An Invariant Characterization of
Pseudo-Valuations}, Proc. Camp. Phil. Soc. 50 (1954), 159-177.


\bibitem [End], O. Endler, \textit {Valuation Theory}. Springer-Verlag,
New York, 1972.



\bibitem [FKK], A. Fornasiero, F.V. Kuhlmann and S. Kuhlmann,
\emph{Towers of complements to valuation rings and truncation
closed embeddings of valued fields}, J. Algebra 323 (2010), no. 3,
574-600.


\bibitem [Hu], J. A. Huckaba, Extensions of Pseudo-Valuations.
\textit {Pacific J. Math.} \textbf{29}, Number 2 (1969), 295--302.



\bibitem [KO], B. G. Kang and D. Y. Oh, \emph{Lifting up an infinite chain of prime
ideals to a valuation ring}, Proc. Amer. Math. Soc. 126 (1998), no. 3, 645–646.


\bibitem [KZ], M. Knebusch and D. Zhang, \textit {Manis Valuations and
Pr\"ufer Extensions}. Springer-Verlag, Berlin, 2002.



\bibitem [Mat], K. Mathiak \emph {Valuations of Skew Fields and Projective Hjelmslev Spaces}, Lecture
Notes in Math., Springer, Berlin, vol. 1175 (1986).



\bibitem [MH], M. Mahadavi-Hezavehi, Matrix Pseudo-Valuations on
Rings and their Associated Skew Fields. \textit {Int. Math. J.}
\textbf{2} (2002), no. 1, 7--30.




\bibitem [Mor], P. J. Morandi, \emph {Value functions on central simple
algebras}, Trans. Amer. Math. Soc., 315 (1989), 605-622.


\bibitem [Wa], A.R. Wadsworth, Valuation theory on finite
dimensional division algebras. In \textit{Valuation
Theory and its Applications}, Vol. 1, eds. F.-V. Kuhlmann et al.,
Fields Inst. Commun. 32, American Mathematical Society,
Providence, RI, 2002.



\end{thebibliography}
\end{document}